\documentclass[a4paper,dvips,12pt,oneside]{article}
\usepackage[makeroom]{cancel}
\usepackage{graphicx}
\usepackage{epstopdf}
\usepackage{subfig}
\usepackage{color}
\usepackage{float}
\usepackage{leftidx}
\usepackage{bigints}
\usepackage{amssymb,amsmath,dsfont,url,epsfig}
\usepackage[left=1.5in, right=0.91in, top=1.1in, bottom=1.2in, includefoot, headheight=13.6pt]{geometry}
\usepackage[T1]{fontenc}
\usepackage{titlesec, blindtext, color}
\definecolor{gray75}{gray}{0.75}

\newcommand{\sln}{\linespread{1}}
\newcommand*{\email}[1]{\href{mailto:#1}{\nolinkurl{#1}} } 
\titleformat{\chapter}[block]{\LARGE\bfseries\sln}{Chapter \thechapter}{11pt}{\newline\huge\bfseries}
\newtheorem{theorem}{Theorem}[section]
\newtheorem{definition}{Definition}[section]

\newenvironment{proof}{\paragraph{Proof:}}{\hfill$\square$}
\newtheorem{Lemma}{Lemma}[section]
\newtheorem{proposition}{Proposition}[section]

\begin{document}
\title{ On equivalence of two non-Riemannian curvatures in warped product Finsler metrics}
\author{Ranadip Gangopadhyay, Anjali Shriwastawa, Bankteshwar Tiwari\\
DST-CIMS, Banaras Hindu University, Varanasi-221005, India}
\maketitle

\begin{abstract}
In this paper we study warped product Finsler metrics  and show that the notion of isotropic $E$-curvature and isotropic $S$-curvature are equivalent for this class of metrics.
\end{abstract}

\section{Introduction}
\label{Sec:1}
The  warped product Riemannian manifold  was introduced by R. L. Bishop and B. O'Neill in 1964 \cite{BO}, to construct a class of complete Riemannian manifolds of negative curvature. Warped product Riemannian manifolds  are the most natural and fruitful generalization of Riemannian products of two manifolds. The notion
of warped products plays very important roles not only in geometry but
also in mathematical physics, especially in general relativity. In fact, many basic solutions of the Einstein field equations, including the Schwarzschild solution and the Robertson-Walker models, are warped product manifolds. The famous John Nash's embedding theorem published in 1956 implies that every warped product Riemannian manifold manifold can be realized as a warped product submanifold in a suitable Euclidean space (\cite{JN1}, \cite{JN2} ).

Later on, the warped product metric was extended to the case of Finsler manifolds by the work of Chen etal. and Kozma et al. (\cite{BCZL},\cite{KPC} ). Recently, some significant progress has been made in the study of Finsler warped product metrics ( \cite{BCZL}, \cite{KPC}, \cite{LM}). It has been observed in (\cite{BCZL}) that spherically symmetric Finsler metrics have warped product structure. However, there are lot of Finsler warped product metrics that are not spherically symmetric and several authors studied those metrics (\cite{BG}, \cite{BTG}, \cite{BR1}, \cite{CY}). In (\cite{MST}), the authors obtain the differential equation that characterizes the spherically symmetric Finsler metrics with vanishing Douglas curvature. Furthermore they obtain all the spherically symmetric Douglas metrics by solving this equation.

A Finsler metric on a smooth manifold is a smoothly varying family of Minkowski norms, one on each tangent space, rather than a family of inner products one on each tangent space, as in the case of Riemannian metrics. It turns out that every Finsler metric induces an inner product in each direction of a tangent space at each point of the manifold. However, in sharp contrast to the Riemannian case, these Finsler-inner products do not only depend on where we are, but also in which direction we are looking. \\
In this paper we consider the warped product Finsler metrics on the manifold $M= I \times \bar{M}$ which is a simple generalization of the Riemannian version,  where $I$ is an open interval of $\mathbb{R}$ and $(\bar{M}, \bar{\alpha})$ is an $(n-1)$ dimensional Riemannian manifold.

There are several non-Riemannian quantities in the Finsler literature, for instance, Cartan torsion, $S$-curvature, $E$-curvature, $\Xi$-curvature, $H$-curvature etc. These quantities become zero for a Riemannian metric. The concept of $S$-curvature was introduced by Z. Shen {\cite{ZS97}} and it becomes  a very important quantity in both local and global Finsler geometry (\cite{ZS3}, \cite{HZ}). The $E$-curvature is  another quantity in Finsler geometry which is associated to the $S$-curvature.\\
The Randers metric is one of the most simplest classes of non-Riemannian Finsler metric that can be written in the form $F= \alpha + \beta$, where $\alpha$ is a Riemannian metric and $\beta$ is a one-form.  X. Cheng and Z. Shen first calculated the $S$-curvature of the Randers metrics and proved that for the Randers metrics the isotropic $S$-curvature is equivalent to the isotropic $E$-curvature \cite{XCZS}.  Therefore, it is very natural to ask, is there any relation between isotropic $S$-curvature and isotropic $E$-curvature for the other classes of Finsler metrics such as spherically symmetric Finsler metrics and more generally for the warped product Finsler metrics $?$  In this paper we prove the following result:
\begin{theorem}\label{th01}
		Any warped product Finsler metric $F=\bar{\alpha} \phi(r,s)$ has isotropic $S$-curvature if and only if it has isotropic $E$-curvature.
	\end{theorem}
\section{Preliminaries}
Let $ M $ be an $n$-dimensional smooth manifold. $T_{u}M$ denotes the tangent space of $M$
 at $u$. The tangent bundle of $ M $ is the disjoint union of tangent spaces $ TM:= \sqcup _{u \in M}T_uM $. We denote the elements of $TM$ by $(u,v)$ where $v\in T_{u}M $ and $TM_0:=TM \setminus\left\lbrace 0\right\rbrace $. \\
 \begin{definition}
 \cite{SSZ} A Finsler metric on $M$ is a function $F:TM \to [0,\infty)$ satisfying the following conditions:
 \\(i) $F$ is smooth on $TM_{0}$,
 \\(ii) $F$ is a positively 1-homogeneous on the fibers of tangent bundle $TM$,
 \\(iii) The Hessian of $\frac{F^2}{2}$ with element $g_{ij}=\frac{1}{2}\frac{\partial ^2F^2}{\partial v^i \partial v^j}$ is positive definite on $TM_0$.\\
 The pair $(M,F)$ is called a Finsler space and $g_{ij}$ is called the fundamental tensor.
 \end{definition}
 
 Let us consider the product manifold $M= I \times \bar{M}$, where $I$ is an interval of $\mathbb{R}$ and $\left( \bar{M},\bar{\alpha}\right) $ is an $(n-1)$ dimensional Riemannian manifold. Let $\left\lbrace {\theta^a}\right\rbrace _{a=2}^n$ be a local coordinate system on $\bar{M}$. Then $\left\lbrace {u^i}\right\rbrace _{i=1}^n$ gives us a local coordinate on $M$ by setting $u^1=r$ and $u^a=\theta^a$. The indices $i,j,k,...$ are ranging from $1$ to $n$ while $a,b,c,...$ are ranging from $2$ to $n$. A vector $v$ on $M$ can be written as $v=v^i\partial/\partial u^i$ and its  projection on $\bar{M}$ is denoted by $\bar{v}=v^a\partial/\partial u^a=v^a\partial/\partial \theta^a$.   A warped product Finsler metric can be written in the form 
 \begin{equation}\label{eqn2}
 F=\bar{\alpha}\sqrt{w(s,r)},
 \end{equation}
 where $w$ is a suitable function defined on an open subset of $\mathbb{R}^2$ and $s=v^1/\bar{\alpha}$. It can be rewritten as \cite{BCZL} 
 \begin{equation}\label{eqn1}
 F=\bar{\alpha}\phi(s,r), \qquad \textnormal{where} \quad \phi(s,r)= \sqrt{w(s,r)}.
 \end{equation}
In \cite{BCZL} Shen et al. also showed that warped product Finsler metrics contain the class of spherically symmetric Finsler metrics.\\
 The coefficients of fundamental metric tensor of the warped product Finsler metrics are given by
 $$\begin{pmatrix}
 g_{11} & g_{1j} \\
 g_{i1} & g_{ij}
 \end{pmatrix}
 = \begin{pmatrix}
 \frac{1}{2}w_{ss} & \frac{1}{2}\chi_s\bar{\alpha}_{v^j} \\
 \frac{1}{2}\chi_s\bar{\alpha}_{v^i} & \frac{1}{2}\chi \bar{a}_{ij}-\frac{1}{2}s\chi_s\bar{\alpha}_{v^i}\bar{\alpha}_{v^j}
 \end{pmatrix}$$
 where, $\chi := 2w-sw_s$ and $\chi_s := w_s-sw_{ss}$.\\
 By some simple calculation we have, $\det(g_{ij})=\frac{1}{2^{n-1}}\chi^{n-2}\Lambda$, where $\Lambda=2ww_{ss}-w_s^2$.\\ 
 Which can be rewritten as
 \begin{equation}\label{eqn18.5}
 \det(g_{ij})= \phi^{n+1}\phi_{ss}(\phi-s\phi_s)^{n-2}
 \end{equation}
 The spray coefficients of the Finsler metric $F$ are defined by 
 \begin{equation}
 G^i= \frac{1}{4}g^{il}\left\lbrace [F^2]_{v^mv^l}v^m-[F^2]_{u^l} \right\rbrace, 
 \end{equation}
 
  \begin{Lemma} [\cite{BCZL}]
  The spray coefficients of warped product Finsler metrics  $F$ are given by
       \begin{equation}\label{eqn10.00}
     G^1=\Phi\bar{\alpha}^2,\quad G^k=\bar{G}^k+\Psi\bar{\alpha}v^k
        \end{equation}
        where
        \begin{equation}\label{eqn10.0}
        \Phi=\frac{1}{4} \left\lbrace \left(  W_r-\chi_r \right) U+s\chi_r V\right\rbrace, \quad \Psi=\frac{1}{4} \left\lbrace \left( W_r-\chi_r \right) V + s\chi_r \left( W + X\right)  \right\rbrace 
        \end{equation}
        and
        \begin{equation}
        \chi= 2w-sw_s, \quad \Lambda = 2ww_{ss}-w_s^2, \quad U=\frac{2\chi-2s\chi_s}{\Lambda}, \quad V=-\frac{2\chi_s}{\Lambda} ,\quad W= \frac{2}{\Lambda} ,\quad  X=\frac{2w_s\chi_s}{\chi \Lambda}.
       \end{equation}
        
      \end{Lemma}
 \begin{definition}\textnormal{The $E$-curvature of a Finsler metric $F$ is defined as 
       \begin{equation}\label{eqn6}
       E_{ij}:=\frac{1}{2}S_{v^iv^j}(u,v)=\frac{1}{2}\frac{\partial^2}{\partial v^i \partial v^j}\left(\frac{\partial G^m}{\partial v^m} \right).
       \end{equation}
        where $G^i$ are the spray coefficients of the Finsler metric $F$.\\
        The Finsler  metric $F$ is said to be of isotropic $E$-curvature if there exists a scalar function $c=c(u)$ on $M$ such that
        \begin{equation}\label{eqn7}
        E_{ij}=\frac{n+1}{2}c(u)F_{v^iv^j}.
        \end{equation}}
       
  \end{definition}

\section{Isotropic $E$-curvature and Isotropic $S$-curvature of the warped product Finsler metrics}
    In Finsler geometry two volume forms are well known, namely  Busemann-Hausdorff volume form (\cite{ZS97}, \cite{SSZ}) and Holmes-Thompson volume form. In a local coordinate system $(u^i,v^i)$, the Busemann-Hausdorff volume form is defined as $dV_{BH}=\sigma_{BH}(u)du$, where
   \begin{equation}\label{eqn9}
   \sigma_{BH}(u)=\dfrac{Vol(B^n(1))}{Vol\left\lbrace (v^i)\in \mathbb{R}^n : F(u,v^i\frac{\partial}{\partial u^i})< 1 \right\rbrace }.
   \end{equation}
   
   \par In a local coordinate system $(u^i, v^i)$, the Holmes-Thompson volume form is defined as $dV_{HT}=\sigma_{HT}(u)du$, where
   \begin{equation}\label{eqn10}
   \sigma_{HT}(u)=\dfrac{1}{Vol(B^n(1)}\int\limits_{F(u,v^i\frac{\partial}{\partial u^i})< 1}det(g_{ij}(u,v)dv).
   \end{equation}
   \begin{definition}
   For a vector $v \in T_uM \setminus \left\lbrace 0 \right\rbrace$, let $\gamma = \gamma(t)$ be the geodesic with $\gamma(0)=u$ and $\dot{\gamma}(0)=v$. The $S$-curvature of the Finsler metric $F$ with volume form $dV=\sigma_F(u)du$ is defined by
   \begin{equation*}
   S(u,v)=\frac{d}{dt}\left[\tau_F(\gamma(t),\dot{\gamma}(t)\right]|_{t=0}, 
   \end{equation*}
   where $\tau_F$ is called distortion of the Finsler metric $F$ and defined by $\tau_F=\log\frac{\sqrt{det(g_{ij})}}{\sigma_F}$.
   \end{definition}
   \begin{proposition}\cite{SSZ}
   In a standard  local coordinate system in $TM$, the $S$-curvature of a Finsler metric $F$ can also be written as
      \begin{equation}\label{eqn8}
      S=\frac{\partial G^m}{\partial v^m}(u, v)-v^m \frac{\partial\left(\log \sigma_F(u)\right) }{\partial u^m}.
      \end{equation}
   \end{proposition}
  \begin{definition}
  A Finsler metric $F$ is said to be of isotropic $S$-curvature if
  \begin{equation}
  S=(n+1)c(u)F
  \end{equation}
  \end{definition}

Before proving the results we need the followings:
\begin{equation}\label{eqn11}
\frac{\partial{\bar{\alpha}}^2}{\partial v^1} = 0, \quad  \frac{\partial{\bar{\alpha}}^2}{\partial v^j} = 2v_j \quad \forall j=2,3,...,n.
 \end{equation}
 \begin{equation}\label{eqn11.0}
s_{v^1}=\frac{1}{\bar{\alpha}}, \quad s_{v^1}v^1=s, \quad s_{v^j}= -\frac{sv_j}{{\bar{\alpha}}^2}, \quad  s_{v^j} v^j=0,\quad \forall j=2,3,...,n.
 \end{equation}
 \begin{equation}\label{eqn12}
{\bar{\alpha}}_{v^1v^j} =0, \quad {\bar{\alpha}}_{v^iv^j} =\frac{{\bar{\alpha}}^2a_{ij}-v_iv_j}{{\bar{\alpha}}^3}, \quad s_{v^1v^j}=-\frac{v_j}{{\bar{\alpha}}^3}, \quad s_{v^iv^j}=\frac{3sv_iv_j-s{\bar{\alpha}}^2a_{ij}}{{\bar{\alpha}}^4}.
 \end{equation}
	\begin{proposition}
		Let $F = \bar{\alpha} \phi(r,s)$ be a warped product Finsler metric on an $n$-dimensional manifold $M = I\times \bar{M}$. Then the $E$-curvature
		of $F$ is given by
				\begin{equation}\label{eqn13}
				E_{11}=\frac{1}{\bar{\alpha}}\left[ \Phi_{sss}+n\Psi_{ss}\right], \quad E_{1j}=E_{j1}=-\frac{s}{\bar{\alpha}}v_j[\Phi_{sss}+n\Psi_{ss}],
				\end{equation}
		\begin{equation}\label{eqn14}
		E_{ij}=\frac{a_{ij}}{\bar{\alpha}}\left\lbrace (\Phi_s-s\Phi_{ss})+n(\Psi-s\Psi_s)\right\rbrace +\frac{v_iv_j}{\bar{\alpha}^3}\left[ ns^2\Psi_{ss}+s^2\Phi_{sss}+s(\Phi_{ss}+n\Psi_s)-(\phi_s+n\Psi)\right], 
		\end{equation}
where $i,j \ne 1$.
\end{proposition}
\begin{proof}
From \eqref{eqn10.00} we have
\begin{equation}\label{eqn14.0}
\frac{\partial G^1}{\partial v^1}=\bar{\alpha}\Phi_s
\end{equation}
and for $k,m\ne 1$,
	\begin{equation}\label{eqn14.1}
    \frac{\partial}{\partial v^m}(\Psi\bar{\alpha} v^k)=\Psi_s s_{v^m}\bar{\alpha}v^k+\Psi\bar{\alpha}_{v^m} v^m+(n-1)\Psi\bar{\alpha}\delta_m^k.
	\end{equation}
	Now putting $k=m$ in \eqref{eqn14.1} and taking summation over $m$ we get
	\begin{equation*}
	\frac{\partial}{\partial v^m}(\Psi \bar{\alpha}v^m)=\Psi\bar{\alpha}_{v^m}v^m+(n-1)\Psi\bar{\alpha}=n\Psi\bar{\alpha}
	\end{equation*}
and therefore, we have
    \begin{equation}\label{eqn15}
    \sum\frac{\partial G^m}{\partial v^m}=\bar{\alpha}[\Phi_s+n\Psi].
    \end{equation}
          Differentiating \eqref{eqn15} with respect to $v^i$ gives
          	\begin{equation}\label{eqn17}
          	\frac{\partial}{\partial v^i}\left( \frac{\partial G^m}{\partial v^m}\right) =\bar{\alpha}_{v^i}[\Phi_s+n\Psi]+\bar{\alpha}[\Phi_{ss}s_{v^i}+n\Psi_s s_{v^i} ].
          	\end{equation}
          	Again differentiating \eqref{eqn17} with respect to $v^j$ gives
          		\begin{equation}
          		\begin{split}
          		\frac{\partial}{\partial v^j}\frac{\partial}{\partial v^i}\left(\frac{\partial G^m}{\partial v^m} \right) =\bar{\alpha}_{v^i} \bar{\alpha}_{v^j}[\Phi_s+n\Psi]+\left(\bar{\alpha}_{v^i}s_{v^j} +\bar{\alpha}_{v^j}s_{v^i} 
          		\right) \left[\Phi_{ss}+n\Psi_s\right]\\+\bar{\alpha}s_{v^i}s_{v^j}\left[\Phi_{sss}+n\Psi_{ss} \right]+\bar{\alpha}s_{v^iv^j}\left[\Phi_{ss}+n\Psi_s\right].
          \end{split}
           \end{equation}
              Putting $i=j=1$ and from \eqref{eqn11}, \eqref{eqn11.0}, \eqref{eqn12}
              \begin{equation}
              E_{11}=\frac{1}{\bar{\alpha}}\left[\Phi_{sss}+n\Psi_{ss}\right] 
              \end{equation}
              and for $j\ne 1$
              \begin{equation}
              E_{1j}=E_{j1}=\bar{\alpha}s_{v^j}\frac{1}{\bar{\alpha}}\left[\Phi_{sss}+n\Psi_{ss} \right]+\bar{\alpha}\frac{\partial}{\partial v^j}\frac{1}{\bar{\alpha}}\left[\Phi_{ss}+n\Psi_s \right]  
              \end{equation}
              and for $i,j \ne 1$ we have
              \begin{equation}\label{eqn17.0}
              E_{ij}=\frac{a_{ij}}{\bar{\alpha}}\left[ \left( \Phi_s-s\Phi_{ss}\right)+n(\Psi-s\Psi_s) \right] +\frac{v_iv_j}{\bar{\alpha}}\left[ ns^2\Psi_{ss}+s^2\Phi_{sss}+s(\Phi_ss+n\Psi_s)-(\Phi_s+n\Psi)\right]. 
              \end{equation} 
	\end{proof}

	\begin{theorem}\label{th1}
The warped product Finsler metric $F=\bar{\alpha} \phi(r,s)$ has isotropic $E$-curvature if and only if 
	\begin{equation}\label{eqn18.1}
\left( \Phi_s-s\Phi_{ss}\right) +n \left( \Psi-s\Psi_s\right) = \kappa\left( \phi-s\phi_s\right) ,
	\end{equation}
	where $\Phi$, $\Psi$  are given in \eqref{eqn10.0} and $\kappa \ne 0$ is a scalar function on $M$.
	\end{theorem}
\begin{proof}
Differentiating \eqref{eqn1} with respect to $v^i$, we have 
  \begin{equation}\label{eqn18.4}
  F_{v^i}={\bar{\alpha}}_{v^i}\phi+{\bar{\alpha}} \phi_ss_{v^i}.
  \end{equation}
  Differentiating again \eqref{eqn18.4} with respect to $v^j$ yields
  \begin{equation}\label{eqn18.0}
  F_{v^iv^j}={\bar{\alpha}}_{v^iv^j}\phi+({\bar{\alpha}}_{v^i}s_{v^j}+{\bar{\alpha}}_{v^j}s_{v^i})\phi_s+{\bar{\alpha}} s_{v^i}s_{v^j}\phi_{ss}+{\bar{\alpha}} s_{v^iv^j}\phi_s.
  \end{equation}
  If $i=j=1$ then we have,
  \begin{equation}\label{eqn19}
  F_{v^1v^1}=\frac{\phi_{ss}}{{\bar{\alpha}}}
  \end{equation}
  and if $i=1, j \ne 1$, or, $i\ne 1, j=1$ then 
  \begin{equation}\label{eqn20}
  F_{v^1v^j}=-\frac{v_j}{\bar{\alpha}^2}s\phi_{ss} \quad \textnormal{and} \quad F_{v^jv^1}=-\frac{v_j}{\bar{\alpha}^2}s\phi_{ss}.
  \end{equation}
  And for $i,j \ne 1$ plugging \eqref{eqn11}, \eqref{eqn12} into \eqref{eqn18.0} yields
  \begin{equation}
    F_{v^iv^j}=\frac{1}{{\bar{\alpha}}^3}\left[ (\phi -s\phi_s){\bar{\alpha}}^2a_{ij}-(\phi-s\phi_s-s^2\phi_{ss})v_iv_j\right]. 
    \end{equation}
In the view of equations \eqref{eqn7}, \eqref{eqn13} and \eqref{eqn14} the warped product Finsler metric is of isotropic $E$-curvature if and only if we have,
\begin{equation}\label{eqn18.02}
\Phi_{sss}+n\Psi_{ss}=\frac{n+1}{2}c(x)\phi_{ss}, 
\end{equation}
	\begin{equation}\label{eqn18.2}
\left( \Phi_s-s\Phi_{ss}\right) +n \left( \Psi-s\Psi_s\right) =\frac{n+1}{2}c(x) (\phi-s\phi_s),
	\end{equation}	
	\begin{equation}\label{eqn18.3}
ns^2\Psi_{ss}+s^2\Phi_{sss}+s(\Phi_ss+n\Psi_s)-(\Phi_s+n\Psi) = \frac{n+1}{2}c(x)(\phi-s\phi_s-s^2\phi_{ss}).
	\end{equation}
	Now we will show that \eqref{eqn18.02} and \eqref{eqn18.3}  can be obtained from	\eqref{eqn18.2}. Differentiating \eqref{eqn18.2} with respect to $s$ we get \eqref{eqn18.02}. Again multiplying \eqref{eqn18.02} by $s^2$ and substract it from \eqref{eqn18.2} we obtain \eqref{eqn18.3}.
	Therefore, $F$ has isotropic $E$-curvature if and only if \eqref{eqn18.1} holds.
\end{proof}
\begin{Lemma}\label{lm1}
Let $F=\bar{\alpha}\phi(s,r)$,  be a warped product Finsler metric on an $n$-dimensional manifold $M = I\times\bar{M}$. Then the volume form  $dV$ is given by $dV=k(r)dV_{\alpha}$ where
\begin{equation}
k(r)=
    \begin{cases}
        \dfrac{\int\limits_{-1}^{1}\frac{(1-s^2)^{\frac{n-3}{2}}}{\phi^n(r,s)}ds}{\int\limits_{0}^{\pi}\sin^{n-2}(t)dt} & if, \quad dV=dV_{BH} \\
        \dfrac{\int\limits_{0}^{\pi}(\sin^{n-2}t) \upsilon(r^2,r\cos t)dt}{\int\limits_{0}^{\pi}\sin^{n-2}(t)dt} & if, \quad dV=dV_{HT} \\
    \end{cases}
\end{equation}
and $dV_{\alpha} =\sqrt{det(a_{ij})}dx$ denotes the Riemannian  volume form of $\alpha$.
\end{Lemma}
\begin{proof}
Fix an arbitrary point $x_0\in U\subset \mathbb{R}^n,$ consider an orthogonal basis at $x_0$ with respect to the Riemannian metric  $\alpha$ so that
\begin{equation*}
\alpha=\sqrt{\sum_{i=1}^n(v^i)^2}.
\end{equation*}
Then the volume form $dV_{\alpha}=\sigma_{\alpha}du$ at the point $x_0$ is given by
$dV_\alpha=\sigma_\alpha$ $du$ at $x_0$.
\begin{equation*}
\sigma_\alpha=\sqrt{det(a_{ij})}=1.
\end{equation*}
Let us consider the coordinate transformation: $\mu:(s, z^a) \to (v^i)$ such that
\begin{equation}
v^1=\frac{s}{\sqrt{1-s^2}}\bar{\alpha} \quad \textnormal {and} \quad v^a= z^a,
\end{equation}
where $\bar{\alpha}=\sqrt{\sum\limits_{a=2}^{n}(z^a)^2}$, $1 \le i,j,k \le n $, and, $2 \le a,b,c \le n$.\\
Then $\alpha=\frac{1}{\sqrt{1-s^2}}\bar{\alpha}$.\\
Therefore,
$$F=\alpha\phi(r,s)=\frac{\bar{\alpha} \phi}{\sqrt{1-s^2}}$$
and the Jacobian of the transformation $\mu:(s, z^a) \to (v^i)$ is given by $\dfrac{\bar{\alpha}}{(1-s^2)^{3/2}}$.\\
Then 
\begin{equation*}
\begin{split}
Vol\left\lbrace (v^i)\in \mathbb{R}^n : F(u,v^i\frac{\partial}{\partial u^i})< 1 \right\rbrace &= \int_{F(u,v)<1}dv \\ & = \int_{\alpha\phi(r,s)<1}dv
\\ & = \int_{\frac{\bar{\alpha}\phi(r,s)}{\sqrt{1-s^2}}<1} \frac{1}{(1-s^2)^{3/2}}\bar{\alpha}ds du \\ &= \int\limits_{-1}^{1}\frac{1}{(1-s^2)^\frac{3}{2}}\left[\int_{\bar{\alpha}<\frac{\sqrt{1-s^2}}{\phi(r,s)}}\alpha du\right] ds\\ &= \frac{1}{n}Vol(S^{n-2})\int\limits_{-1}^{1}\frac{1}{(1-s^2)^\frac{3}{2}}\left( \frac{\sqrt{1-s^2}}{\phi(r,s)}\right) ^nds\\ &=\frac{1}{n} Vol(S^{n-2})\int\limits_{-1}^{1}\frac{(1-s^2)^{\frac{n-3}{2}}}{\phi^n(r,s)}ds\\&=\frac{1}{n}Vol(S^{n-2})k(r)
\end{split}
\end{equation*}
where $$k(r)=\int\limits_{-1}^{1}\frac{(1-s^2)^{\frac{n-3}{2}}}{\phi^n(r,s)}ds.$$
Since, $Vol(B^n(1))=\frac{1}{n}Vol(S^{n-2})\int\limits_{0}^{\pi}\sin^{n-2}(t)dt$, we have 
\begin{equation}
\sigma_{BH}=\frac{\int\limits_{-1}^{1}\frac{(1-s^2)^{\frac{n-3}{2}}}{\phi^n(r,s)}ds}{\int\limits_{0}^{\pi}\sin^{n-2}(t)dt}.
\end{equation} 

Let us consider $\upsilon=\phi\phi_{ss}(\phi-s\phi_s)^{n-2}$. Hence, from \eqref{eqn10} and \eqref{eqn18.5} we have,
\begin{equation}
\begin{split}
\int_{F(u,v^i\frac{\partial}{\partial u^i}<1)}det(g_{ij}(u, v))dv &=\int_{F(u,v)<1}\phi^n(r,s)\upsilon(r,s)ds\\ &=\frac{1}{n}Vol(S^{n-2})\int\limits_{-1}^{1}(1-s^2)^{(n-3)/2}\upsilon(r,s)ds \\ &= \frac{1}{n}Vol(S^{n-2})\int\limits_{0}^{\pi}(\sin^{n-2}t) \upsilon(r^2,r\cos t)dt.
\end{split}
\end{equation}
Therefore, 
\begin{equation}
\sigma_{HT}=\frac{\int\limits_{0}^{\pi}(\sin^{n-2}t) \upsilon(r^2,r\cos t)dt}{\int\limits_{0}^{\pi}\sin^{n-2}(t)dt}.
\end{equation}
Hence we have the result.
\end{proof}
\begin{theorem}
A  warped product Finsler metric $F=\bar{\alpha} \phi(r,s)$ on a manifold $M = I\times\bar{M}$ has isotropic $S$-curvature with respect to  volume form $dV_{BH}$ or $dV_{HT}$ if and only if
\begin{equation}\label{eqn23}
\Phi_s+n\Psi+g(r)s=(n+1)\phi c(u),
\end{equation}
where $c(u)$ is a scalar function on $M$ and $g(r)=-r\frac{k'(r)}{k(r)}$.
\end{theorem}
\begin{proof}
Since, $\bar{G}^k=\frac{1}{2}\Gamma^k_{ij}v^iv^j$,  we have 
\begin{equation}\label{eqn23.0}
[\bar{G}^k]_{v^k}=\Gamma^k_{kj}v^j= v^m \frac{\partial}{\partial u^m}(\log \sigma_{\bar{\alpha}}).
\end{equation}
	By Lemma \eqref{lm1}, we have $dV=\sigma du= K(r)\sigma_{\bar{\alpha}}du$. Hence,
	\begin{equation}\label{}
	v^m\frac{\partial}{\partial u^m}(\log \sigma)= \dfrac{k'(r)}{k(r)}v^m\dfrac{\partial r}{\partial u^m}+v^m\dfrac{\partial}{\partial u^m}(\log \sigma_{\bar{\alpha}}).
	\end{equation}
The $S$-curvature of a warped product Finsler metric  $F=\bar{\alpha} \phi(r,s)$ is given by 
\begin{equation}\label{eqn24}
S= v^m \frac{\partial}{\partial u^m}(\log \sigma_{\bar{\alpha}})+\bar{\alpha} \left[ \Phi_s+n\Psi\right] - v^m\dfrac{\partial}{\partial u^m}(\log \sigma).
\end{equation}
Again 
\begin{equation}\label{eqn24.0}
 v^m \frac{\partial}{\partial u^m}(\log \sigma_{\bar{\alpha}})= \bar{\alpha} s r.
\end{equation}
Now, from \eqref{eqn8}, \eqref{eqn24} and \eqref{eqn24.0} we obtain
\begin{equation} \label{}
S= \bar{\alpha} \left[ \Phi_s+n\Psi+g(r)s\right],
\end{equation}
where
\begin{equation}\label{eqn25}
g(r)=-r\dfrac{k'(r)}{k(r)}.
\end{equation}
Therefore, $F$ has isotropic $S$-curvature if and only if \eqref{eqn23} holds.
\end{proof}\\

\textbf{Proof of Theorem \ref{th01}}

	Let the Finsler metric $F$ has isotropic $S$-curvature. Then \eqref{eqn23} holds. Therefore, differentiating \eqref{eqn23} with respect to $s$ we have,
		\begin{equation}\label{eqn26}
		\Phi_{ss}+n\Psi_s+g(r)=(n+1)c(u)\phi_s.
		\end{equation}
		Now multiplying \eqref{eqn26} by $s$ and substracting it from \eqref{eqn23} gives \eqref{eqn18.1}. Hence, $F$ has isotropic $E$-curvature.\\
	Conversely, suppose $F$ has isotropic $E$-curvature. Then \eqref{eqn18.1} holds. As $s \ne 0$, it can be rewritten as
	\begin{equation}
	\frac{s\Phi_{ss}-\Phi_s}{s^2}+n\frac{s\Psi_s-\Psi}{s^2}=(n+1)c(u)\frac{\phi-s\phi_s}{s^2},
	\end{equation}
	which implies
		\begin{equation}
	\frac{\partial}{\partial s}\left(\frac{\Phi_s}{s} \right)  + n\frac{\partial}{\partial s}\left(\dfrac{\Psi}{s} \right) = (n+1)c(u)\frac{\partial}{\partial s}\left( \frac{\phi_s}{s}\right).
		\end{equation}
		Integrating we have,
		\begin{equation}\label{eqn27}
\frac{\Phi_s}{s} + n\frac{\Psi}{s}+g(r) = (n+1)c(u)\frac{\phi_s}{s},
		\end{equation}
		where $g(r)$ is the constant of integration.\\
		 Now in particular, if we choose $g(r)=-r\dfrac{k'(r)}{k(r)}$ and multiplying \eqref{eqn27} by $s$ we get \eqref{eqn23}. Therefore, $F$ has isotropic $S$-curvature.

\end{document}